\documentclass[11pt,letterpaper,reqno]{amsart}

\usepackage[T1]{fontenc}
\usepackage{amsmath,amssymb,amsfonts,amsthm}
\usepackage{mathtools}
\usepackage{aliascnt}
\usepackage{microtype}
\usepackage{tikz}
\usetikzlibrary{positioning}
\usepackage{xcolor}
\usepackage{doi}
\usepackage{hyperref}
\usepackage{bookmark}
\usepackage[nameinlink, capitalize,noabbrev]{cleveref}

\hypersetup{
  pdfstartview={FitH},
  colorlinks=true,
  linkcolor=blue,
  citecolor=blue,
  urlcolor=blue
}

\newtheorem{theorem}{Theorem}[section]

\newaliascnt{lemma}{theorem}
\newtheorem{lemma}[lemma]{Lemma}
\aliascntresetthe{lemma}

\newaliascnt{proposition}{theorem}
\newtheorem{proposition}[proposition]{Proposition}
\aliascntresetthe{proposition}

\newaliascnt{corollary}{theorem}
\newtheorem{corollary}[corollary]{Corollary}
\aliascntresetthe{corollary}

\newaliascnt{conjecture}{theorem}
\newtheorem{conjecture}[conjecture]{Conjecture}
\aliascntresetthe{conjecture}

\newaliascnt{problem}{theorem}

\aliascntresetthe{problem}

\theoremstyle{definition}
\newaliascnt{definition}{theorem}

\aliascntresetthe{definition}

\newaliascnt{example}{theorem}

\aliascntresetthe{example}

\newaliascnt{remark}{theorem}

\aliascntresetthe{remark}

\crefname{theorem}{Theorem}{Theorems}
\Crefname{theorem}{Theorem}{Theorems}
\crefname{lemma}{Lemma}{Lemmas}
\Crefname{lemma}{Lemma}{Lemmas}
\crefname{proposition}{Proposition}{Propositions}
\Crefname{proposition}{Proposition}{Propositions}
\crefname{corollary}{Corollary}{Corollaries}
\Crefname{corollary}{Corollary}{Corollaries}
\crefname{conjecture}{Conjecture}{Conjectures}
\Crefname{conjecture}{Conjecture}{Conjectures}

\newcommand{\cG}{\mathcal{G}}

\begin{document}

\title{The $\Delta$-Conjecture for CIS $d$-Graphs}

\author[Y.~Liu]{Yinchen Liu}
\address{Institute for Interdisciplinary Information Sciences, Tsinghua University, Beijing 100084, P. R. China}
\email{liuyinch23@mails.tsinghua.edu.cn}

\author[Q.~Tang]{Quanyu Tang}
\address{School of Mathematical Sciences, University of Science and Technology of China, Hefei 230026, P. R. China}
\email{tangquanyu827@gmail.com}

\begin{abstract}
We prove the $\Delta$-conjecture, which dates back to Gurvich's 1978 thesis. Specifically, let the edges of a complete graph be colored with colors $1,\ldots,d$, and for each $i$ let $G_i$ be the graph on the same vertex set formed by the edges of color $i$. We prove that if every choice of a maximal stable set $S_i$ of $G_i$, one for each $i\in[d]$, has nonempty intersection, then the coloring contains no rainbow triangle. Together with a result of Andrade, Boros, and Gurvich, this characterizes CIS $d$-graphs as precisely the Gallai $d$-graphs whose chromatic components are ordinary CIS graphs. We also show that every factor in the canonical modular decomposition of a CIS $d$-graph is a CIS $d$-graph whose edge-coloring uses at most two colors.
\end{abstract}

\subjclass[2020]{Primary 05C15; Secondary 05C69}
% 05C15  	Coloring of graphs and hypergraphs
% 05C69  	Vertex subsets with special properties (dominating sets, independent sets, cliques, etc.)

\keywords{CIS graph, Gallai coloring, rainbow triangle,
$\Delta$-conjecture, modular decomposition}

\maketitle

\section{Introduction}

Throughout the paper, all graphs are finite and simple, and maximality is with
respect to inclusion.  For a graph \(H=(W,F)\), a set \(S\subseteq W\) is
\emph{stable} (or \emph{independent}) if
\(\binom{S}{2}\cap F=\varnothing\).  It is a maximal stable set if and only if
every vertex of \(W\setminus S\) has a neighbor in \(S\).

A graph is called \emph{CIS} if every maximal clique meets every maximal
stable set.  The property appears under the name \emph{clique--kernel
intersection property} in the survey of Brandst\"adt, Le, and Spinrad
\cite[pp.~175--176]{BrandstadtLeSpinrad1999}; a systematic account of CIS graphs and their
multicolored extensions was later given by Andrade, Boros, and Gurvich
\cite{AndradeBorosGurvich2006,AndradeBorosGurvich2018}.  To distinguish
this notion from the multicolored generalization recalled below, we refer
to it as the \emph{ordinary} CIS property.

The structural theory of ordinary CIS graphs has developed
in several directions.  Berge posed two early problems, both resolved by
Zang \cite{Berge1985,Zang1995}.  A conjecture of Chv\'atal giving a local
characterization in terms of extensions of induced \(P_4\)'s was proved by
Deng, Li, and Zang and, independently, by Andrade, Boros, and Gurvich
\cite{DengLiZang2004,DengLiZang2005,AndradeBorosGurvich2006,
AndradeBorosGurvich2018}.  Further work includes the characterization of
almost CIS graphs \cite{BorosGurvichZverovich2009,WuZangZhang2009}, the
study of CIS graphs in relation to split, equistable, and other
clique--stable-set classes \cite{BorosGurvichMilanic2017}, and structural
results for vertex-transitive and claw-free CIS graphs
\cite{DobsonHujdurovicMilanicVerret2015,AlconGutierrezMilanic2019}.

The complexity of recognizing ordinary CIS graphs was a
long-standing open problem, posed by Chv\'atal in the 1990s
\cite{TaoYangZang2026}.  It led to two competing conjectures.
Andrade, Boros, and Gurvich conjectured that CIS graphs can be
recognized in polynomial time
\cite{AndradeBorosGurvich2006,AndradeBorosGurvich2018},
whereas Zverovich and Zverovich conjectured that the recognition
problem is \(\mathsf{coNP}\)-complete
\cite{ZverovichZverovich2006}.  Tao, Yang, and Zang very recently
settled the problem by proving that recognizing ordinary CIS graphs
is \(\mathsf{coNP}\)-complete \cite{TaoYangZang2026}.

We next recall the multicolored setting.  For an integer \(d\geq 2\), write
\([d]=\{1,\ldots,d\}\).  A \emph{\(d\)-graph} is a
complete graph whose edges are colored with colors in \([d]\); formally, it
is a pair
\[
  \cG=(V;\varphi),
  \qquad
  \varphi:\binom{V}{2}\longrightarrow[d].
\]
The coloring need not use every color.  For each \(i\in[d]\), let
\[
  E_i=\varphi^{-1}(i)
  \qquad\text{and}\qquad
  G_i=(V,E_i).
\]
The graphs \(G_1,\ldots,G_d\) are the \emph{chromatic components} of
\(\cG\).  Thus, \(S\subseteq V\) is stable in \(G_i\) if and only if
\(\varphi(uv)\neq i\) for all distinct \(u,v\in S\); it is a maximal stable
set if and only if, in addition, every vertex of \(V\setminus S\) has an
\(i\)-colored neighbor in \(S\).

The \(d\)-graph \(\cG\) is called \emph{CIS} if, for every choice
of a maximal stable set \(S_i\) of \(G_i\), one for each \(i\in[d]\), one has
\[
  \bigcap_{i=1}^{d}S_i\neq\varnothing.
\]
When \(d=2\), the chromatic components are complementary, so this definition
recovers the usual CIS property.  Notice also that the intersection of one
stable set from each chromatic component has size at most one: if two distinct
vertices belonged to all the sets, then the edge joining them would have some
color \(i\), contradicting stability in \(G_i\). Thus, in a CIS \(d\)-graph, the intersection of any choice of one
maximal stable set from each chromatic component consists of exactly
one vertex. This
observation underlies the blocker and anti-blocker formulations developed by
Gurvich \cite{Gurvich2011}.

A triangle is \emph{rainbow} if its three edges have pairwise distinct
colors.  An edge-coloring of a complete graph with no rainbow triangle is
called a \emph{Gallai coloring}, after Gallai's foundational work
\cite{Gallai1967}; see Gy\'arf\'as and Simonyi \cite{GyarfasSimonyi2004} for a
modern graph-theoretic treatment.  We call a \(d\)-graph \emph{Gallai} when
its coloring is a Gallai coloring.

The following conjecture is traced to Gurvich's 1978 thesis
\cite{Gurvich1978}.  It is commonly called the \(\Delta\)-conjecture.

\begin{conjecture}\label{conj:delta}
Every CIS \(d\)-graph is Gallai.
\end{conjecture}

The conjecture has motivated a substantial body of structural work.  A
reduction attributed to Andrey Gol'berg in 1975 already shows that it
suffices to treat three colors, by merging colors into three groups
\cite[Section~3.10]{GurvichNaumova2026}.  Building on this reduction,
Gurvich studied modular decompositions of complete edge-colored graphs
\cite{Gurvich2009}, Andrade, Boros, and Gurvich classified the locally
minimal non-CIS \(d\)-graphs \cite{AndradeBorosGurvich2010}, and blocker and
anti-blocker methods were developed in \cite{Gurvich2011}.  A comprehensive
account, including the local settling configurations around rainbow
triangles that we use below, appears in
\cite{AndradeBorosGurvich2006,AndradeBorosGurvich2018}, where Steven
Jaslar's computer verification of the three-color case for all
\(3\)-graphs on at most twelve vertices is also reported
\cite[Section~1.6]{AndradeBorosGurvich2018}.  Despite this body of work, the
conjecture remained open as recently as the survey of Gurvich and Naumova
\cite[Section~3.10]{GurvichNaumova2026}.

We prove the conjecture in the following equivalent, contrapositive form,
from which Conjecture~\ref{conj:delta} follows at once.

\begin{theorem}\label{thm:main}
Let \(\cG=(V;\varphi)\) be a \(d\)-graph.  If \(\cG\) contains a rainbow triangle,
then there exist maximal stable sets \(S_i\) of \(G_i\), one for each
\(i\in[d]\), such that
\[
  \bigcap_{i=1}^{d}S_i=\varnothing.
\]
\end{theorem}

Two ingredients of the proof are already present in the literature:
the reduction to three colors by merging colors
\cite[Proposition~6]{AndradeBorosGurvich2018}, and the settling
configuration around a rainbow triangle
\cite[Section~4.2]{AndradeBorosGurvich2018}. In the CIS setting, the latter yields two additional vertices with forced incident colors and, regardless of the
color of the edge joining them, produces a new rainbow triangle.
The key new ingredient is Lemma~\ref{lem:restriction}, which allows this
local configuration to be used inductively: after deleting an appropriate
color-neighborhood, one obtains a proper induced subgraph that still
contains a rainbow triangle, and empty-intersection maximal stable sets
in the smaller graph can be lifted to the original graph.

Our theorem also completes the reduction of the multicolored
CIS property to the ordinary one.  Indeed, combining it with
\cite[Proposition~11]{AndradeBorosGurvich2018} gives the exact
characterization
\[
  \cG\text{ is a CIS \(d\)-graph}
  \quad\Longleftrightarrow\quad
  \begin{cases}
    \cG\text{ is Gallai},\\
    G_i\text{ is an ordinary CIS graph for every }i\in[d];
  \end{cases}
\]
see Proposition~\ref{prop:component-characterization}.  Thus the
multicolored CIS property introduces no further condition beyond a
Gallai coloring and the ordinary CIS property of its chromatic components.

As a further structural consequence, combining our result
with the canonical modular decomposition of symmetric \(2\)-structures
gives every CIS \(d\)-graph a canonical modular-decomposition tree whose
internal-node quotients use at most two colors and, viewed as ordinary
graphs, are CIS; see
Corollary~\ref{cor:modular-decomposition}.

The paper is organized as follows.  Section~\ref{sec:restriction} proves
the restriction lemma that drives the induction.  Section~\ref{sec:three}
establishes the three-color case by induction, using the settling
configuration around a rainbow triangle.  Section~\ref{sec:general} reduces
the general case to three colors by color merging and completes the proof
of Theorem~\ref{thm:main}, then records its structural consequences.

\section{The restriction lemma}\label{sec:restriction}

We begin with the restriction lemma that drives the induction. The projection argument of
Section~\ref{sec:general} and the settling configuration used in
Section~\ref{sec:three} are both already present in the literature.  For
\(r\in[d]\) and \(Z\subseteq V\), write
\[
  N_r(Z):=\{v\in V\setminus Z: vz\in E_r\text{ for some }z\in Z\}.
\]
For \(A\subseteq V\), let \(\cG[A]\) denote the induced \(d\)-graph
obtained by restricting \(\varphi\) to \(\binom{A}{2}\).

\begin{lemma}\label{lem:restriction}
Fix \(r\in[d]\). Let \(Z\) be a nonempty stable set of \(G_r\), and put $A=V\setminus N_r(Z)$. Suppose that there exist maximal stable sets \(I_i\) of the \(i\)-th
chromatic component of \(\cG[A]\), one for each \(i\in[d]\), such that $\bigcap_{i=1}^{d}I_i=\varnothing$. Then there exist maximal stable sets \(T_i\) of \(G_i\), one for each
\(i\in[d]\), such that
\[
  \bigcap_{i=1}^{d}T_i=\varnothing.
\]
\end{lemma}

\begin{proof}
Every vertex of \(Z\) is isolated in the subgraph of \(G_r\) induced by
\(A\).  Hence \(Z\subseteq I_r\).  Moreover, \(I_r\) is already a maximal
stable set of \(G_r\).  Indeed, every vertex of \(A\setminus I_r\) has an
\(r\)-colored neighbor in \(I_r\), by maximality inside \(A\), whereas every
vertex of \(V\setminus A\) has an \(r\)-colored neighbor in
\(Z\subseteq I_r\), by the definition of \(A\).

For each \(i\neq r\), extend \(I_i\) to a maximal stable set \(T_i\) of
\(G_i\), and set \(T_r=I_r\).  The maximality of \(I_i\) in the subgraph of
\(G_i\) induced by \(A\) implies that every vertex of \(A\setminus I_i\) has
an \(i\)-colored neighbor in \(I_i\).  Consequently,
\[
  T_i\cap A=I_i
  \qquad\text{for every }i\neq r.
\]
If a vertex \(v\) belonged to every \(T_i\), then
\(v\in T_r=I_r\subseteq A\), and hence \(v\in T_i\cap A=I_i\) for every
\(i\neq r\).  This would put \(v\) in
\(\bigcap_{i=1}^{d}I_i\), a contradiction.
\end{proof}

\section{The three-color case}\label{sec:three}

We now prove the essential case of the conjecture.

\begin{theorem}\label{thm:three}
Let \(\cG=(V;\varphi)\) be a \(3\)-graph containing a rainbow triangle.  Then
there exist maximal stable sets \(S_i\) of \(G_i\), one for each
\(i\in[3]\), such that
\[
  S_1\cap S_2\cap S_3=\varnothing.
\]
\end{theorem}

\begin{proof}
We argue by induction on \(n=|V|\).  If \(n=3\), then, after relabeling the
vertices and colors, we may write
\[
  \varphi(ab)=1,
  \qquad
  \varphi(bc)=2,
  \qquad
  \varphi(ca)=3.
\]
The sets \(\{b,c\}\), \(\{c,a\}\), and \(\{a,b\}\) are then maximal stable
sets of \(G_1\), \(G_2\), and \(G_3\), respectively, and their intersection
is empty.

Assume that \(n>3\) and that the result holds for every smaller \(3\)-graph
containing a rainbow triangle.  Choose a rainbow triangle \(abc\) and relabel
its vertices and colors so that
\[
  \varphi(ab)=1,
  \qquad
  \varphi(bc)=2,
  \qquad
  \varphi(ca)=3.
\]
Extend \(\{b,c\}\), \(\{c,a\}\), and \(\{a,b\}\) to maximal stable sets
\(X_1\), \(X_2\), and \(X_3\) of \(G_1\), \(G_2\), and \(G_3\),
respectively (note that in $G_1,G_2,G_3$ the sets $\{b,c\}$, $\{c,a\}$, and $\{a,b\}$ are independent).  If \(X_1\cap X_2\cap X_3=\varnothing\), we are done.
Otherwise, let \(x\in X_1\cap X_2\cap X_3\). Necessarily \(x\notin\{a,b,c\}\): indeed,
\(a\notin X_1\), \(b\notin X_2\), and \(c\notin X_3\). The stability of the three
sets forces
\[
  \varphi(xa)=1,
  \qquad
  \varphi(xb)=2,
  \qquad
  \varphi(xc)=3.
\]

\begin{figure}[htbp]
\centering
\begin{tikzpicture}[
  scale=1.05,
  every node/.style={font=\small},
  vertex/.style={
    circle,
    draw=black,
    line width=0.8pt,
    fill=white,
    minimum size=20pt,
    inner sep=1.5pt
  },
  colorone/.style={
    draw=red,
    line width=1.2pt
  },
  colortwo/.style={
    draw=green!55!black,
    line width=1.2pt
  },
  colorthree/.style={
    draw=blue,
    line width=1.2pt
  }
]

%---------------- vertices ----------------
\node[vertex] (a) at (0,2.1) {$a$};
\node[vertex] (b) at (-2.35,0) {$b$};
\node[vertex] (c) at (2.35,0) {$c$};
\node[vertex] (x) at (-4.15,2.45) {$x$};
\node[vertex] (y) at (4.15,2.45) {$y$};

%---------------- main rainbow triangle abc ----------------
\draw[colorone]   (a) -- (b);
\draw[colortwo]   (b) -- (c);
\draw[colorthree] (c) -- (a);

%---------------- edges from x ----------------
\draw[colorone]   (x) -- (a);
\draw[colortwo]   (x) -- (b);
\draw[colorthree] (x) -- (c);

%---------------- edges from y ----------------
\draw[colorthree] (y) -- (a);
\draw[colorone]   (y) -- (b);
\draw[colortwo]   (y) -- (c);

%---------------- xy after cyclic relabeling ----------------
\draw[colorone] (x) to[bend left=22] (y);

%---------------- centered legend ----------------
\draw[colorone]   (-3.9,-1.55) -- (-2.8,-1.55);
\node[anchor=west] at (-2.55,-1.55) {color \(1\)};

\draw[colortwo]   (-0.95,-1.55) -- (0.15,-1.55);
\node[anchor=west] at (0.40,-1.55) {color \(2\)};

\draw[colorthree] (2.00,-1.55) -- (3.10,-1.55);
\node[anchor=west] at (3.35,-1.55) {color \(3\)};

\end{tikzpicture}
\caption{The settling configuration after a cyclic relabeling, so that
\(\varphi(xy)=1\); the triangle \(cxy\) is again rainbow. Colors \(1,2,3\)
are red, green, and blue, respectively.}
\label{fig:settling}
\end{figure}
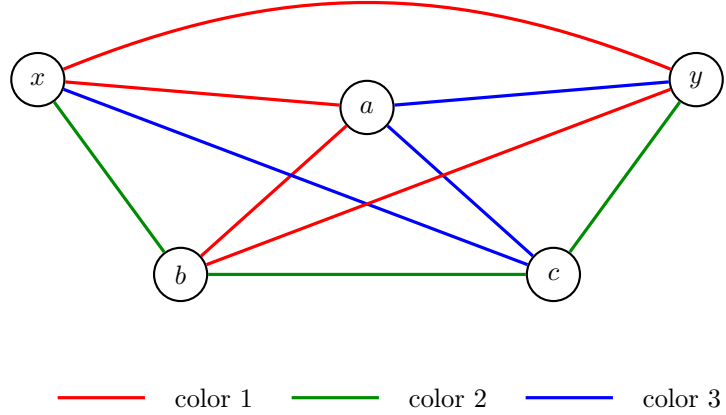

Next extend \(\{a,c\}\), \(\{a,b\}\), and \(\{b,c\}\) to maximal stable
sets \(Y_1\), \(Y_2\), and \(Y_3\) of \(G_1\), \(G_2\), and \(G_3\),
respectively.  Again, there is nothing to prove if
\(Y_1\cap Y_2\cap Y_3=\varnothing\).  Otherwise, let
\(y\in Y_1\cap Y_2\cap Y_3\). Necessarily \(y\notin\{a,b,c\}\): indeed, \(b\notin Y_1\), \(c\notin Y_2\), and \(a\notin Y_3\). Stability now forces
\[
  \varphi(ya)=3,
  \qquad
  \varphi(yb)=1,
  \qquad
  \varphi(yc)=2.
\]
In particular \(x\neq y\). Cyclically relabeling \((a,b,c)\mapsto(b,c,a)\) together with colors
\(2,3,1\mapsto1,2,3\) preserves the whole configuration and sends
\(\varphi(xy)\) through \(1\mapsto3\mapsto2\mapsto1\); applying it at most
twice, we may assume
\[
  \varphi(xy)=1,
\]
so the triangle \(cxy\) is rainbow. The forced colors are shown in
Figure~\ref{fig:settling}.

Set \(A=V\setminus N_2(\{x\})\). Since \(\varphi(xb)=2\), we have
\(b\notin A\), so \(A\) is a proper subset of \(V\); since \(\varphi(xc)=3\)
and \(\varphi(xy)=1\), we have \(c,x,y\in A\). Thus \(\cG[A]\) is a smaller
\(3\)-graph containing the rainbow triangle \(cxy\), so by the induction
hypothesis its three chromatic components have maximal stable sets with
empty intersection. Lemma~\ref{lem:restriction}, applied with \(Z=\{x\}\)
and \(r=2\), lifts these to the required maximal stable sets of
\(G_1,G_2,G_3\).
\end{proof}

\section{The general case and structural consequences}\label{sec:general}
Now, we prove the general case by a known color merging reduction.

Let \(2\le k\le d\), and let $[d]=P_1\mathbin{\dot\cup}\cdots\mathbin{\dot\cup}P_k$ be a partition into nonempty sets, and define \(p:[d]\to[k]\) by
\(p(i)=j\) whenever \(i\in P_j\).  The \emph{projection} of
\(\cG=(V;\varphi)\) associated with this partition is the \(k\)-graph
\[
  \cG'=(V;p\circ \varphi).
\]
Projections of CIS \(d\)-graphs are CIS
\cite[Proposition~6]{AndradeBorosGurvich2018}, a known fact for which we
record the following explicit contrapositive form.

\begin{lemma}\label{lem:projection}
For each \(j\in[k]\), let \(S_j\) be a maximal stable set of the \(j\)-th
chromatic component of \(\cG'\).  If
\(\bigcap_{j=1}^{k}S_j=\varnothing\), then there exist maximal stable sets \(T_i\) of \(G_i\), one for each
\(i\in[d]\), such that
\[
  \bigcap_{i=1}^{d}T_i=\varnothing.
\]
\end{lemma}

\begin{proof}
For each \(j\in[k]\) and each \(i\in P_j\), the set \(S_j\) is stable in
\(G_i\), because it contains no edge whose original color lies in \(P_j\).
Extend \(S_j\) to a maximal stable set \(T_i\) of \(G_i\).

Suppose that a vertex \(v\) belongs to every \(T_i\).  Fix \(j\in[k]\).  If
\(v\notin S_j\), then the maximality of \(S_j\) in the \(j\)-th chromatic
component of \(\cG'\) gives a vertex \(u\in S_j\) such that
\(p(\varphi(uv))=j\).  Hence \(\varphi(uv)=i\) for some \(i\in P_j\).  Since
\(u\in S_j\subseteq T_i\) and \(v\in T_i\), this contradicts the stability
of \(T_i\) in \(G_i\).  Therefore \(v\in S_j\) for every \(j\), contrary to
\(\bigcap_{j=1}^{k}S_j=\varnothing\).
\end{proof}

We can now complete the proof of the main theorem.

\begin{proof}[Proof of Theorem~\ref{thm:main}]
Let the three colors on a rainbow triangle of \(\cG=(V;\varphi)\) be
\(\alpha\), \(\beta\), and \(\gamma\).  Partition \([d]\) into three
nonempty sets \(P_1,P_2,P_3\) such that
\(\alpha\in P_1\), \(\beta\in P_2\), and \(\gamma\in P_3\), and merge all
colors in each part. The resulting projection is a \(3\)-graph containing
the same rainbow triangle.  By Theorem~\ref{thm:three}, there exist
maximal stable sets \(S_1,S_2,S_3\) of its three chromatic components,
respectively, such that $S_1\cap S_2\cap S_3=\varnothing$. Lemma~\ref{lem:projection} lifts \(S_1,S_2,S_3\) to maximal stable sets of
\(G_1,\ldots,G_d\) whose intersection is empty. 
\end{proof}

We conclude with two structural consequences of Theorem~\ref{thm:main}.
Both were previously noted, conditional on the $\Delta$-conjecture; see
\cite[Sections~3.1 and~3.3]{Gurvich2009} and
\cite[Sections~3.9--3.10]{GurvichNaumova2026}.
For completeness, we include their unconditional derivations from
Theorem~\ref{thm:main} and known results.
\begin{proposition}\label{prop:component-characterization}
Let \(\cG=(V;\varphi)\) be a \(d\)-graph, with chromatic
components \(G_i=(V,E_i)\), \(i\in[d]\).  Then \(\cG\) is CIS if and only if
\(\cG\) is Gallai and every \(G_i\) is an ordinary CIS graph.
\end{proposition}

\begin{proof}
If \(\cG\) is CIS, then it is Gallai by
Theorem~\ref{thm:main}; \cite[Proposition~11]{AndradeBorosGurvich2018} then implies that every \(G_i\) is CIS.
Conversely, if \(\cG\) is Gallai and every \(G_i\) is CIS, the same
proposition implies that \(\cG\) is CIS.
\end{proof}

We next recall the modular decomposition as follows.
A nonempty set \(M\subseteq V\) is a \emph{module} of a
\(d\)-graph \(\cG=(V;\varphi)\) if, for every \(x\in V\setminus M\), the
color \(\varphi(xu)\) is independent of \(u\in M\).  A module is
\emph{strong} if it overlaps no other module, where two sets overlap if
they intersect but neither contains the other.  The strong modules form a
unique rooted tree under inclusion, with root \(V\) and singleton leaves
\cite{EhrenfeuchtRozenberg1990}. The tree is called the \emph{canonical
modular decomposition tree}. If \(M\) is an internal node, its children
partition \(M\). Since each child is a module, all edges between any two
children have the same color. We may therefore contract every child to one
vertex, obtaining a well-defined quotient \(d\)-graph \(Q_M\).
These quotients are the \emph{factors} of the canonical modular
decomposition.

\begin{corollary}\label{cor:modular-decomposition}
Every CIS \(d\)-graph has a unique canonical modular decomposition
whose factors are CIS \(d\)-graphs using at most two colors.
\end{corollary}

\begin{proof}
By Theorem~\ref{thm:main}, every CIS \(d\)-graph is Gallai.
Every Gallai \(d\)-graph admits a modular decomposition into \(2\)-graphs
\cite[Theorem~A]{GyarfasSimonyi2004}; the underlying fact is implicit in
Gallai \cite[Lemma~(3.2.3)]{Gallai1967}. We claim that the factors in the
\emph{canonical} modular decomposition also use at most two colors. Every
nonmonochromatic canonical factor has only trivial modules, namely the singletons and the whole vertex set~\cite{EhrenfeuchtRozenberg1990}. Suppose that such a factor \(Q\) used at least three colors. Since every
factor of a Gallai \(d\)-graph is again Gallai, the Gallai decomposition
theorem gives a partition of \(V(Q)\) into modules whose quotient uses at
most two colors. Since \(Q\) itself uses at least three colors, not all
parts of this partition can be singletons. Hence \(Q\) has a nontrivial
module, contradicting its primitivity. Therefore
every canonical factor uses at most two colors.

If \(Q\) has vertices \(q_1,\ldots,q_t\), the substitution
\(Q[H_1,\ldots,H_t]\) is obtained by replacing each \(q_j\) with \(H_j\),
retaining the colors within each \(H_j\), and assigning every edge between
\(H_j\) and \(H_k\) the color of \(q_jq_k\) in \(Q\).
The CIS property is preserved in both directions under substitution:
\(Q[H_1,\ldots,H_t]\) is CIS if and only if \(Q\) and all \(H_j\) are CIS
\cite[Propositions~7 and~10]{AndradeBorosGurvich2018}. Applying this
equivalence at every internal node of the canonical tree shows that all
factors are CIS. Since a \(2\)-graph is just an ordinary graph together with
its complement, its CIS property is exactly the ordinary CIS property.
\end{proof}

\section*{Acknowledgments}
We thank Vladimir Gurvich for helpful comments, for pointing out the
modular-decomposition consequence recorded in
Corollary~\ref{cor:modular-decomposition}, and for drawing our attention
to the relevant earlier references.

\section*{Statement on AI usage}
The main proof strategy, in particular the inductive use of the local configuration around a rainbow triangle together with restriction to a proper induced subgraph, was developed by the authors. ChatGPT was used during the preparation of the manuscript to assist with some technical details, including parts of the proof of Lemma~\ref{lem:restriction}, and with the preparation of Figure~\ref{fig:settling}. It was also used to help check and polish some arguments. All AI-assisted arguments were independently verified and revised by the authors. The authors take full responsibility for all content in this paper.

\end{document}